\documentclass{amsart}

\newcommand{\apref}[3]{\hyperref[#2]{#1\ref*{#2}#3}}

\usepackage{enumerate}
\usepackage[latin1]{inputenc}
\usepackage{dsfont}
\usepackage{amssymb,amsthm,amsmath}

\input{xy}
\xyoption{all}

\usepackage{mathrsfs}

\theoremstyle{plain}
\newtheorem{prop}{Proposition}[section]
\newtheorem{lemma}[prop]{Lemma}

\newtheorem{thm}[prop]{Theorem}

\newtheorem*{thmAnn}{Theorem A}
\newtheorem*{thmBnn}{Theorem B}

\theoremstyle{definition}

\theoremstyle{remark}

\DeclareMathOperator{\Primlength}{PL}

\DeclareMathOperator{\base}{base}

\DeclareMathOperator{\PSL}{PSL}
\DeclareMathOperator{\PGL}{PGL}

\DeclareMathOperator{\Tr}{Tr}

\DeclareMathOperator{\Rea}{Re}

\DeclareMathOperator{\Fix}{Fix}

\DeclareMathOperator{\bd}{bd}

\DeclareMathOperator{\Per}{Per}

\newcommand{\st}{\text{st}}

\newcommand\N{\mathbb{N}}

\newcommand\R{\mathbb{R}}
\newcommand\Z{\mathbb{Z}}
\newcommand\C{\mathbb{C}}

\newcommand{\h}{\mathbb{H}}

\newcommand{\mc}[1]{\mathcal #1}

\newcommand{\wt}{\widetilde}
\newcommand{\wh}{\widehat}

\DeclareMathOperator{\id}{id}

\DeclareMathOperator{\Fct}{Fct}

\newcommand{\sceq}{\mathrel{\mathop:}=}

\newcommand{\mat}[4]{\begin{pmatrix} #1&#2\\#3&#4\end{pmatrix}}
\newcommand{\bmat}[4]{\begin{bmatrix} #1&#2\\#3&#4\end{bmatrix}}

\newcommand{\textbmat}[4]{\left[\begin{smallmatrix} #1&#2 \\ #3&#4
\end{smallmatrix}\right]}

\usepackage[colorlinks,breaklinks]{hyperref}
\usepackage{graphics}
\usepackage[T1]{fontenc} 
\usepackage{pdflscape}
\allowdisplaybreaks

\begin{document}

\title[Thermodynamic formalism for infinite-area Hecke triangle surfaces]{A thermodynamic formalism approach to the Selberg zeta function for Hecke triangle surfaces of infinite area}
\author[A.\@ Pohl]{Anke D.\@ Pohl}
\address{Mathematisches Institut, Georg-August-Universit\"at G\"ottingen,  Bunsenstr. 3-5, 37073 G\"ottingen}
\email{pohl@uni-math.gwdg.de}
\begin{abstract} 
We provide an explicit construction of a cross section for the geodesic flow on infinite-area Hecke triangle surfaces which allows us to conduct a transfer operator approach to the Selberg zeta function. Further we construct closely related cross sections for the billiard flow on the associated triangle surfaces and endow the arising discrete dynamical systems and transfer operator families with two weight functions which presumably encode Dirichlet respectively Neumann boundary conditions. The Fredholm determinants of these transfer operator families constitute dynamical zeta functions, which provide a factorization of the Selberg zeta function of the Hecke triangle surfaces.
\end{abstract}
\thanks{The author acknowledges the support by the Volkswagen Foundation}
\maketitle


\section{Introduction}

Selberg zeta functions are important objects in the study of the spectral theory of Riemannian locally symmetric spaces (or, more generally, orbifolds). They were introduced by Selberg \cite{Selberg} in 1956 for compact quotients of the hyperbolic plane, motivated by his study of automorphic forms for uniform Fuchsian groups. Subsequently these zeta functions were intensively studied and generalized to non-compact quotients, general rank one spaces and even higher rank spaces. They contributed significantly to the cross-fertilization of ideas in various subject areas, ranging from analytic number theory to quantum chaos. We refer to \cite{Venkov_book, Fischer, Hejhal1, Hejhal2, Deitmar_zeta_higher, Borthwick_book} and the references therein for more details. 

Since these dynamical zeta functions are defined by an infinite product over the length spectrum of the space under consideration which only converges in some half space, a crucial step in their investigation is to show the existence of a meromorphic continuation. For Selberg zeta functions of hyperbolic Riemannian orbifolds of the form $X = \Gamma\backslash\h$,
where $\h$ denotes the hyperbolic plane and $\Gamma$ is Fuchsian group (possibly being non-torsionfree or non-cofinite), this is typically done by Selberg theory, Lax--Phillips scattering theory or geometric scattering theory. These methods are very elegant and powerful, in particular they allow to continue to develop a rich theory of the properties and interpretations of zeros and poles as well as manifold applications. However, they have the drawback that even for the proof of basic properties a significant amount of theory is needed. 

Over the last 25 years another method emerged within the framework of the thermodynamic formalism of statistical mechanics, as pioneered by Ruelle \cite{Ruelle_formalism, Ruelle_zeta} and Mayer \cite{Mayer_thermo, Mayer_thermoPSL}. By exploiting the dynamics of the geodesic flow on $X$ rather than the geometry of $X$ these so-called transfer operator approaches provide alternative proofs of some findings obtained by the previously mentioned methods, or even complementary results. In particular, the existence of meromorphic continuations is easier to prove. For such a transfer operator approach to the Selberg zeta function it is essential to have a discretization of the geodesic flow which gives rise to a uniformly expanding discrete dynamical system $(D,F)$ on a family of subsets of $\R$ such that the associated family of transfer operators (weighted evolution operators of functions $f\colon D\to\C$)
\[
 \mc L_{F,s} f(x) \sceq \sum_{y\in F^{-1}(x)} |F'(y)|^{-s} f(y)
\]
represents the Selberg zeta function via its Fredholm determinant (we refer to Section~\ref{sec:prelims} for a more details). At the time being, this requirement results in the disadvantage that such approaches are not yet established for arbitrary Fuchsian groups. However, they are known for various cofinite ones (and also for some lattices containing orientation-reversing isometries) \cite{Mayer_thermoPSL, Pollicott, Morita_transfer, Fried_triangle, Mayer_Muehlenbruch_Stroemberg, 
Moeller_Pohl, Pohl_spectral_hecke}. For non-cofinite Fuchsian groups, up to date, they could only be performed for Fuchsian Schottky groups \cite{Guillope_Lin_Zworski,Naud_resonancefree}. In this case, the associated orbifolds have infinite area, no singularity points and no cusps. 

In this article we consider for the first time a family of infinite-area hyperbolic orbifolds with one cusp and one elliptic point, namely the Hecke triangle surfaces of infinite area. Our first main result can roughly be summarized as follows.

\begin{thmAnn}
For any Hecke triangle surface $X$ of infinite area, there exists a discretization of the geodesic flow such that, for $\Rea s > 1/2$, the arising family of transfer operators $\mc L_s$ is nuclear of order $0$ on a certain Banach space of holomorphic functions. Its Fredholm determinant represents the Selberg zeta function
\[ 
 Z_X(s) = \det(1-\mc L_s)
\]
for sufficiently large $\Rea s$. The map $s\mapsto \mc L_s$ extends meromorphically to all of $\C$ with possible poles at $s=(1-k)/2$, $k\in\N_0$, of order at most $2$. Consequently, also $Z_X$ admits a meromorphic continuation with these properties.
\end{thmAnn}

The proof of Theorem~A is constructive. Its main bulk, which we conduct in Section~\ref{sec:systems}, consists in providing an explicit cross section (in the sense of Poincar\'e) for the geodesic flow on $X$, and using it to construct a discrete dynamical system with the required properties. We study the properties of the arising family of transfer operators in Section~\ref{sec:domain}. The combination of Theorems~\ref{nuclear} and \ref{fredholm} below then shows Theorem~A.

We will observe that the cross section is invariant under the orientation-reversing Riemannian isometry $J\colon z\mapsto -\overline z$. This external symmetry is inherited by the discrete dynamical system and the family of transfer operators, and results in a factorization of the Selberg zeta function. To achieve a deeper understanding of this phenomen, we consider the extension of the Hecke triangle group $\Gamma$ by $J$, which gives the underlying triangle group $\wt\Gamma \subseteq \PGL2(\R)$. We then investigate, in Section~\ref{sec:billiard}, the billiard flow on the triangle surface $\wt\Gamma\backslash\h$, modify the cross sections from Theorem~A into cross sections for this flow and endow the arising discrete dynamical systems and transfer operators with two different weight functions, which presumably correspond to Dirichlet resp.\@  Neumann boundary conditions. Our second main result is a more explicit version of the following theorem, being proven as Theorems~\ref{fredtriangle} and \ref{explicit} 
below.

\begin{thmBnn} 
For $\Rea s > 1/2$, the two weighted families $\mc L_s^\pm$ of transfer operators, which arise from the discretization of the billiard flow on $\wt X \sceq \wt\Gamma\backslash\h$, are nuclear operators of order $0$ on a certain Banach space of holomorphic functions, and the maps $s\mapsto \mc L_s^\pm$ extend meromorphically to all of $\C$ with possible poles at $s=(1-k)/2$, $k\in\N_0$, of order at most $1$. Thus, also their Fredholm determinants
\[
 Z_\pm(s) \sceq \det\left(1-\mc L_s^\pm\right)
\]
define meromorphic functions on $\C$. They factorize the Selberg zeta function $Z_X$ as
\[
 Z_X(s) = Z_+(s) Z_-(s).
\]
Moreover, for sufficiently large $\Rea s$, the functions $Z_\pm$ are given as dynamical zeta functions in terms of the length spectrum of $\wt X$.
\end{thmBnn}

Our constructions of the discretizations always involve a first discretization whose associated discrete dynamical systems are non-uniformly expanding with only finitely many preimages of any given point. For this article, these systems are of auxiliary nature, but we expect that the eigenfunctions of their associated transfer operators are intimately related to the resonances of the Laplacian on $X$ resp.\@ $\wt X$, and comment on this in Section~\ref{sec:outlook}. 

This article is part of a program to study spectral properties of Riemannian locally symmetric spaces and orbifolds with transfer operator techniques, see, e.g.,  \cite{Pohl_diss, Hilgert_Pohl, Pohl_Symdyn2d, Moeller_Pohl, Pohl_mcf_Gamma0p, Pohl_mcf_general,   Mayer_thermo, Chang_Mayer_eigen, Chang_Mayer_extension, Efrat_spectral, Deitmar_Hilgert,  Lewis, Lewis_Zagier, BLZ_part2, Pohl_spectral_hecke} and the references already given above. In particular, we expect that the transfer operators obtained here will make possible an investigation of resonances similiar to that in \cite{Borthwick_numerical, Barkhofen_etal, Weich} for Fuchsian Schottky groups.

\section{Preliminaries}\label{sec:prelims}

\subsection{Hyperbolic geometry}
As model for the hyperbolic plane, we use the upper half plane
\[
 \h \sceq \{ z = x + iy \in\C\mid y > 0\}
\]
endowed with the Riemannian metric given by the line element $ds^2 = y^{-2}(dx^2 + dy^2)$. We identify its geodesic boundary with $P^1(\R) \cong \R\cup\{\infty\}$. Let 
\[
 J \sceq \bmat{-1}{0}{0}{1} \in \PGL_2(\R).
\]
The group of Riemannian isometries on $\h$ is isomorphic to 
\[
 \PGL_2(\R) = \PSL_2(\R) \cup J \PSL_2(\R),
\]
whose action on $\h$ extends continuously to $P^1(\R)$. The subgroup $\PSL_2(\R)$ of orientation-preserving isometries acts by fractional linear transformations. Thus, for $\textbmat{a}{b}{c}{d}\in \PSL_2(\R)$ and $z\in\h\cup\R$, we have
\[
\bmat{a}{b}{c}{d}.z = 
\begin{cases}
\frac{az+b}{cz+d} & \text{for $cz+d\not=0$}
\\
\infty & \text{for $cz+d=0$} 
\end{cases}
\quad\text{and}\quad
\bmat{a}{b}{c}{d}.\infty = 
\begin{cases}
\frac{a}{c} & \text{for $c\not=0$}
\\
\infty & \text{for $c=0$.}
\end{cases}
\]
The action of $J$ is given by
\[
 J.z = -\overline z \quad\text{and}\quad J.\infty = \infty.
\]
The action of $\PGL_2(\R)$ obviously induces an action on the unit tangent bundle $S\h$ of $\h$.

\subsection{Hecke triangle groups}\label{sec:Heckeprelims}
The Hecke triangle group $\Gamma_\lambda$ with parameter $\lambda>0$ is the subgroup of $\PSL_2(\R)$ generated by the two elements
\[
 S = \bmat{0}{1}{-1}{0} \quad\text{and}\quad T_\lambda = \bmat{1}{\lambda}{0}{1}.
\]
The group $\Gamma_\lambda$ is Fuchsian (i.e., discrete) if and only if $\lambda\geq 2$ or $\lambda = 2\cos\tfrac{\pi}{q}$ with $q\in\N_{\geq 3}$. A fundamental domain, indicated in Figure~\ref{funddoms}, for a Fuchsian Hecke triangle group $\Gamma_\lambda$ is given by 
\[
 \mc F_\lambda \sceq \left\{ z\in\h\ \left\vert\ |z|>1,\ |\Rea z|< \frac{\lambda}2 \right.\right\},
\]
of which the vertical sides $\{ \Rea z = -\lambda/2\}$ and $\{ \Rea z = \lambda/2\}$ are paired by $T_\lambda$, and the arc-sides $\{ |z|=1, \Rea z \leq 0\}$ and $\{ |z|=1, \Rea z\geq 0\}$ are paired by $S$. 

\begin{figure}[h]
\begin{center}
\includegraphics{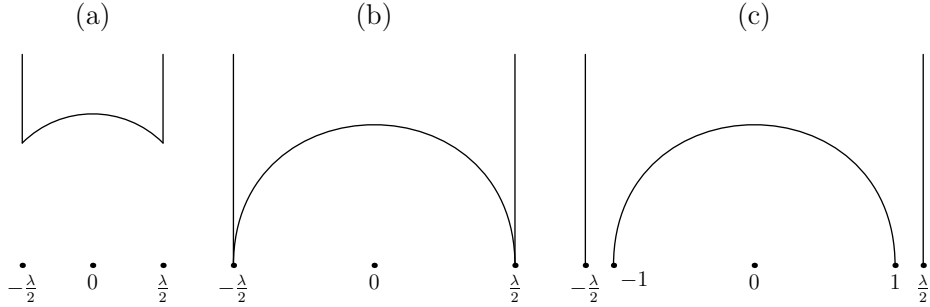}
\end{center}
\caption{Fundamental domains for $\Gamma_\lambda$ if (a) $\lambda<2$, (b) $\lambda=2$, and (c) $\lambda>2$.}
\label{funddoms}
\end{figure}

For $\lambda=2\cos\tfrac{\pi}{q}$ ($q\in\N_{\geq 3}$) and for $\lambda=2$, the discrete Hecke triangle groups are (non-uniform) Fuchsian lattices. The Hecke triangle group $\Gamma_1$ is the modular group, the lattice $\Gamma_2$ is isomorphic to the projective version of $\Gamma_0(2)$. Thermodynamic formalism approaches to the Selberg zeta functions for these cofinite Hecke triangle groups have been established in \cite{Moeller_Pohl}. In this article, we consider the non-cofinite Hecke triangle groups $\Gamma_\lambda$, $\lambda>2$. The associated infinite-area orbifolds, the Hecke triangle surfaces,
\[
 X_\lambda \sceq \Gamma_\lambda\backslash\h
\]
have one funnel, one cusp and one elliptic point. The funnel is represented by the subset $[-\lambda/2, -1]\cup [1,\lambda/2]$ of $\R$. The cusp is represented by $\infty$ with stabilizer group $\langle T_\lambda\rangle$, and the elliptic point is represented by $i$ with stabilizer group $\{\id, S\}$. We use
\[
 SX_\lambda = \Gamma_\lambda \backslash S\h
\]
to denote the unit tangent bundle of $X_\lambda$. We parametrize all geodesics on $\h$ and on $X_\lambda$ by arc length.

From now on, we shall omit all subscripts $\lambda$. 

\subsection{Selberg zeta function}\label{sec:selberg}

Let $\delta$ denote the Hausdorff dimension of the limit set of $\Gamma$. For $\Rea s > \delta$, the Selberg zeta function of $X$ is defined by
\begin{equation}\label{formulaSelberg}
 Z(s) = \prod_{\ell\in \Primlength} \prod_{k=0}^\infty \left(1-e^{-(s+k)\ell}\right),
\end{equation}
where  $\Primlength$ denotes the primitive geodesic length spectrum of $X$ repeated according to multiplicities.
It is well-known that $Z$ is holomorphic and nonvanishing on $\Rea s > \delta$, that $\delta$ equals the exponent of convergence of the Poincar\'e series for $\Gamma$, that $\delta$ is a zero of $Z$ and $\delta(1-\delta)$ is the largest eigenvalue of the (positive definite) Laplace-Beltrami operator on $X$ \cite{Patterson, Sullivan}, and $\delta > 1/2$ \cite{Beardon_exponent1, Beardon_exponent2}.

\subsection{Dynamical systems and discretizations}\label{def:cross}

We call a subset $\wh C$ of $SX$ a cross section for the geodesic flow on $X$ if and only if the intersection between any geodesic and $\wh C$ is discrete in space and time, and each periodic geodesic intersects $\wh C$ (infinitely often).
Since the Selberg zeta function only depends on the periodic geodesics, we may use here this relaxed notion of cross section. A set of representatives for $\wh C$ is a subset $C'$ of $S\h$ such that the canonical quotient map $\pi\colon S\h \to SX$ induces a bijection $C'\to \wh C$.

For any $\wh v\in SX$ let $\wh\gamma_{\wh v}$ denote the geodesic on $X$ determined by 
\[
 \wh\gamma'_{\wh v}(0) = \wh v.
\]
The first return map of a cross section $\wh C$ is given by 
\[
 R\colon \wh C \to \wh C,\quad \wh v \mapsto \wh\gamma'_{\wh v}(t(\wh v)),
\]
whenever the first return time
\[
 t(\wh v) \sceq \min\{ t>0 \mid \wh\gamma'_{\wh v}(t) \in \wh C\}
\]
exists. 

Let $\wh{\mc V}$ denote the set of geodesics on $X$ which converge to the cusp or the funnel of $X$. Let $T\wh{\mc V}$ denote the set of unit tangent vectors to the elements in $\wh{\mc V}$, and set $T\mc V \sceq \pi^{-1}(T\wh{\mc V})$. For $v\in S\h$ let $\gamma_v$ denote the geodesic on $\h$ determined by 
\[
 \gamma'_v(0) = v.
\]
Let 
\[
 \bd \sceq \{ \gamma_v(\infty) \mid v\in T\mc V\} \quad \subseteq \R \cup \{\infty\}
\]
be the set of endpoints of the geodesics determined by the elements in $T\mc V$. The points in 
\[
 \big(P^1(\R) \times \bd\big) \cup \big(\bd \times P^1(\R)\big)
\]
are then precisely the endpoints of the geodesics in $\pi^{-1}(\wh{\mc V})$. For any subset $I$ of $\R$ we set
\[
 I_\st \sceq I \setminus \bd.
\]

The cross section $\wh C$ we will construct in Section~\ref{sec:systems} are not intersected at all by the geodesics in $\wh{\mc V}$, and by each other geodesic infinitely often both in backward and forward time. Thus, its first return map is defined everywhere. Moreover there exists a set of representatives $C'$ which decomposes into at most countably many sets $C'_\alpha$, $\alpha\in A$, each one consisting of a certain fractal set of unit tangent vectors whose base points form a vertical geodesic on $\h$ and all of who point into the same half space determined by this geodesic. On each $C'_\alpha$, the map $v\mapsto \gamma_v(\infty)$ is injective, and its image is of the form $I_{\alpha,\st}$ for some interval $I_\alpha$ in $\R$. 

Via the map $\tau\colon \wh C \to \R\times A$,
\[
 \tau(\wh v) \sceq (\gamma_v(\infty), \alpha) \qquad \text{for $v=\pi^{-1}(\wh v) \cap C' \in C'_\alpha$},
\]
the first return map $R\colon \wh C \to \wh C$ induces a discrete dynamical system $(D_\st, F)$ on 
\[
 D_\st = \bigcup_{\alpha\in A} I_{\alpha, \st}\times \{\alpha\}.
\]
The special structure of the sets $C'_\alpha$ implies that $F$ decomposes into countably many submaps (bijections) of the form 
\[
 \big( I_{\alpha,\st} \cap g^{-1}_{\alpha,\beta}. I_{\beta,\st}\big) \times \{\alpha\} \to I_{\beta,\st} \times \{\beta\},\quad (x,\alpha) \mapsto (g_{\alpha,\beta}.x, \beta),
\]
where $\alpha,\beta\in A$ and $g_{\alpha,\beta}$ is an element in $\Gamma$. For any $v\in C'_\alpha$ there is a first future intersection between $\gamma_v(\R_{>0})$ and $\Gamma.C'$, say on $g_{\alpha,\beta}.C'_\beta$, which completely determines these submaps. 

Each submap can be continued to an analytic map on the ``analytic hull'' of $I_{\alpha,\st} \cap g^{-1}_{\alpha,\beta}. I_{\beta,\st}$, that is the minimal interval $I$ in $\R$ such that $I_\st = I_{\alpha,\st} \cap g^{-1}_{\alpha,\beta}. I_{\beta,\st}$. We will continue to denote the arising piecewise analytic map by $F$.

Finally, we call a finite sequence $(x_k)_{k=1}^n$ in $D_\st$ (or the domain of a piecewise analytic extension) an $F$-periodic orbit of length $n$ if and only if $F(x_k) = x_{k+1}$ for $k=1,\ldots, n-1$, and $F(x_n) = x_1$. We consider two $F$-periodic orbits as equivalent if and only if they have the same length and are identical after some cyclic permutation.

\subsection{Transfer operators}

Given a discrete dynamical system $(D, F)$, where $D$ is a family of subsets of $\R$, $F$ is differentiable and each point has at most countably many preimages, the associated transfer operator $\mc L_{F, s}$ ($s\in\C$) is (formally) given by
\[
 \big(\mc L_{F,s} f\big)(x) \sceq \sum_{y\in F^{-1}(x)} |F'(y)|^{-s} f(y),
\]
acting on an appropriate space of functions $f\colon D\to\C$ (to be adapted to the system and applications under consideration).

For $s\in\C$, $g=\textbmat{a}{b}{c}{d}\in \PGL_2(\R)$ and $x\in\R$, we set
\[
 j_s(g,x) \sceq \big( |\det g|\cdot (cx+d)^{-2} \big)^s.
\]
For a function $f\colon V\to \R$ on some subset of $V$ of $P^1(\R)$, we define
\[
\tau_s(g^{-1})f(x) \sceq j_s(g,x) f(g.x),
\]
whenever this makes sense.

If the system $(D,F)$ decomposes into countably many submaps of the form 
\begin{equation}\label{submaps_abstract}
 D_\alpha \to F(D_\alpha),\quad x \mapsto g_\alpha.x \qquad (\alpha\in A)
\end{equation}
for some $g_\alpha\in \PGL_2(\R)$, then the (formal) transfer operator $\mc L_{F,s}$ becomes
\[
 \mc L_{F,s} f = \sum_{\alpha\in A} 1_{F(D_\alpha)} \cdot \tau_s(g_\alpha)\big(f\cdot 1_{D_\alpha}\big),
\]
where $1_E$ denotes the characteristic function of the set $E$.

\subsection{Thermodynamic formalism}\label{sec:thermo}

Given a sufficiently ``nice'' discrete dynamical system $(D,F)$ related to the geodesic flow on $X$, the proof that, for large $\Rea s$, the Fredholm determinant of its associated family of transfer operators $\mc L_{F,s}$ represents the Selberg zeta function of $X$ is by now standard (see \cite{Ruelle_formalism, Mayer_thermo}). In this section we briefly recall the structure of this proof. The main purpose of this article is then to construct discretizations with the requested properties. 

Let $\wh C$ be a cross section for the geodesic flow on $X$ with some set $C'$ of representatives. Suppose that $(\wh C, C')$ gives rise, in the way as explained in Section~\ref{def:cross}, to a discrete dynamical system $(D,F)$ such that 
\begin{enumerate}[(i)]
\item it decomposes into at most countably many submaps of the form \eqref{submaps_abstract} with $g_\alpha\in\Gamma$,
\item\label{condbij} the equivalence classes of $F$-periodic orbits are in bijection with the periodic geodesics on $X$, and
\item it is uniformly expanding, which here translates to the requirement that the transfer operators $\mc L_{F,s}$ are nuclear of order $0$ for sufficiently large $\Rea s$.
\end{enumerate}

We use the fact that periodic geodesics on $X$ are in bijection with the conjugacy classes of the hyperbolic elements in $\Gamma$. For a hyperbolic element $a\in\Gamma$, its norm $N(a)$ is the square of the eigenvalue of $a$ with larger absolute value. The periodic geodesic $\wh\gamma$ on $X$ which corresponds to the conjugacy class of $a$ has length 
\[
\ell(\wh \gamma) = \log N(a)
\]
and is represented by the geodesic $\gamma$ on $\h$ whose future resp.\@ past endpoint is the attracting resp.\@ repelling fixed point of $a$: 
\[
 \gamma(+\infty) = \lim_{n\to\infty} a^n.i \quad\text{and}\quad \gamma(-\infty) = \lim_{n\to\infty} a^{-n}.i.
\]
We consider the Smale-Ruelle zeta function for $X$, which, for $\Rea s >\delta$, is given by
\[
 \zeta_{\text{SR}}(s) \sceq \prod_{\ell\in\Primlength} \left(1- e^{-s\ell}\right)^{-1}.
\]
For $n\in\N$ let 
\[
 \Fix R^n \sceq \left\{ \wh v \in \wh C \ \left\vert\  R^n(\wh v) = \wh v \vphantom{ \wh C}  \right.\right\}
\]
be the set of fixed points of $R^n$. The $n$-th dynamical partition function is given by
\[
 Z_n(R,s) \sceq \sum_{\wh v\in\Fix R^n} \exp\left( -s \sum_{k=0}^{n-1} t\big(R^k(\wh v)\big)\right).
\]
Then we have
\[
 \zeta_{\text{SR}}(s) = \exp\left( \sum_{n\in\N} \frac1n Z_n(R,s) \right).
\]
If $(x_k)_{k=1}^n$ is an $F$-periodic orbit, let $(a_k)_{k=1}^n$ denote the associated sequence of acting group elements in the submaps \eqref{submaps_abstract}, that is $a_k.x_k=x_{k+1}$ for $k=1,\ldots, n-1$, and $a_n.x_n=x_1$. Further let $\Per_n$ denote the set of arising sequences $(a_k)_{k=1}^n$, and $P_n$ the elements $a_n\cdots a_1$. Let $\wh v \in \Fix R^n$, and set $x_1 \sceq \tau(\wh v)$. Then $x_1$ starts an $F$-periodic orbit of length $n$ with associated sequence $(a_k)_{k=1}^n$. We set
\[
 a(\wh v) \sceq a_n\cdots a_1.
\]
By \eqref{condbij}, $P_n$ consists only of hyperbolic elements, and 
\[
 \Fix R^n \to P_n,\quad \wh v \mapsto a(\wh v),
\]
is a bijection with 
\[
 \sum_{k=0}^{n-1} t\big(R^k(\wh v)\big) = \log N\big(a(\wh v)\big).
\]
Therefore,
\[
 Z_n(R,s) = \sum_{a\in P_n} N(a)^{-s}.
\]
If $\tau_s(a)$ and $\mc L_{F,s}$ are considered to act on an appropriate Banach space of holomorphic functions (see Section~\ref{sec:domain}), we have
\[
 \Tr \tau_s(a) = \frac{N(a)^{-s}}{1-N(a)^{-1}} \quad\text{and}\quad \Tr \mc L_{F,s}^n = \sum_{a\in P_n} \Tr \tau_s(a).
\]
Then
\[
 Z_n(R,s) = \Tr \mc L_{F,s}^n - \Tr \mc L_{F,s+1}^n.
\]
For the Selberg zeta function it now follows
\begin{align*}
Z(s) & = \prod_{k=0}^\infty \zeta_{\text{SR}}(s+k)^{-1}
\\
& = \prod_{k=0}^\infty \exp\left( \sum_{n\in\N} \frac1n \left( \Tr \mc L_{F,s+k}^n - \Tr \mc L_{F,s+k+1}^n\right)\right)
\\
& = \exp\left( -\sum_{n\in\N} \frac1n \Tr \mc L_{F,s}^n \right) \cdot \lim_{k\to\infty} \exp\left( \sum_{n\in\N} \frac1n \Tr ßmc L_{F,s+k}^n\right) 
\\
&=  \exp\left( -\sum_{n\in\N} \frac1n \Tr \mc L_{F,s}^n \right)
\\
& = \det\left( 1 - \mc L_{F,s} \right).
\end{align*}

\section{Discrete dynamical systems and transfer operators}\label{sec:systems}

The construction of the discrete dynamical system for the thermodynamic formalism approach to the Selberg zeta function of $X$ will be done in three steps. In Section~\ref{sec:firstcross} below we recall the cross section $\wh C_P$, a set of representatives $C'_P$ and the induced discrete dynamical system $(D_P,F_P)$ with finitely many submaps from \cite{Pohl_Symdyn2d}. This system has the advantage that the equivalence classes of $F_P$-periodic orbits are already known to be in bijection with the conjugacy classes of hyperbolic elements in $\Gamma$, and that it is eventually expanding. However, it is not uniformly expanding due to the following three reasons:
\begin{itemize}
\item the appearance of the action by identity in the submaps, which causes a, for our purposes, overly refined cross section,
\item the appearance of the action by elliptic elements in the submaps, which causes a locally contracting behavior, and
\item the way of appearance of the action by parabolic elements in the submaps, which causes a locally non-expanding non-contracting behavior.
\end{itemize}

To overcome the first two issues we reduce $\wh C_P$ and $C'_P$ to certain minimal subsets $\wh C_R$ and $C'_R$ which preserve the essential properties of $\wh C_P$ and $C'_P$. The arising discrete dynamical system $(D_R, F_R)$ still decomposes into only finitely many submaps and hence is only eventually expanding due to the third issue mentioned above. We will comment in Section~\ref{sec:outlook} on a conjectural application of these two systems. 

To eliminate the third issue we apply an induction procedure on $\wh C_R$ and $C'_R$ to construct a cross section $\wh C_I$ with set of representatives $C'_I$ such that the induced discrete dynamical system $(D_I, F_I)$ is uniformly expanding and still enjoys the property that the equivalence classes of its periodic orbits are in bijection with the conjugacy classes of the hyperbolic elements.

\subsection{First symbolic dynamics for the geodesic flow}\label{sec:firstcross}

We let $\base(v) \in \h$  denote the base point of a unit tangent vector $v\in S\h$. The set of representatives $C'_P$ for the cross section $\wh C_P$ from \cite{Pohl_Symdyn2d} decomposes into the disjoint subsets
\begin{align*}
C'_a & \sceq \{ v\in S\h \mid \Rea\base(v) = -1,\ \gamma_v(-\infty) \in \R_\st,\  \gamma_v(\infty) \in (-1,\infty)_\st\},
\\
C'_b & \sceq \{ v\in S\h \mid \Rea\base(v) = 1,\ \gamma_v(-\infty) \in \R_\st,\  \gamma_v(\infty) \in (-\infty, 1)_\st\},
\\
C'_c & \sceq \{ v\in S\h \mid \Rea\base(v) = -1,\ \gamma_v(-\infty) \in \R_\st,\  \gamma_v(\infty) \in (-\infty, -1)_\st\},
\\
C'_d & \sceq \{ v\in S\h \mid \Rea\base(v) =  1,\ \gamma_v(-\infty) \in \R_\st,\  \gamma_v(\infty) \in (1,\infty)_\st\},
\\
C'_e & \sceq \left\{ v\in S\h \,\left\vert\ \Rea\base(v) = -\tfrac{\lambda}{2},\ \gamma_v(-\infty) \in \R_\st,\  \gamma_v(\infty) \in \left(-\tfrac{\lambda}{2},\infty\right)_\st \right.\right\},
\\
C'_f & \sceq \left\{ v\in S\h \,\left\vert\  \Rea\base(v) = \tfrac{\lambda}{2},\ \gamma_v(-\infty) \in \R_\st,\  \gamma_v(\infty)\in \left(-\infty, \tfrac{\lambda}{2}\right)_\st\right.\right\}, \text{ and}
\\
C'_g & \sceq \{ v\in S\h \mid  \Rea\base(v) = 0,\ \gamma_v(-\infty) \in \R_\st,\  \gamma_v(\infty)\in (0,\infty)_\st\}.
\end{align*}

This set of representatives and their relevant $\Gamma$-translates for the determination of the induced discrete dynamical system are indicated in Figure~\ref{setrep_original}. For detailed proofs we refer to \cite{Pohl_Symdyn2d}.

\begin{figure}[h]
\begin{center}
\includegraphics{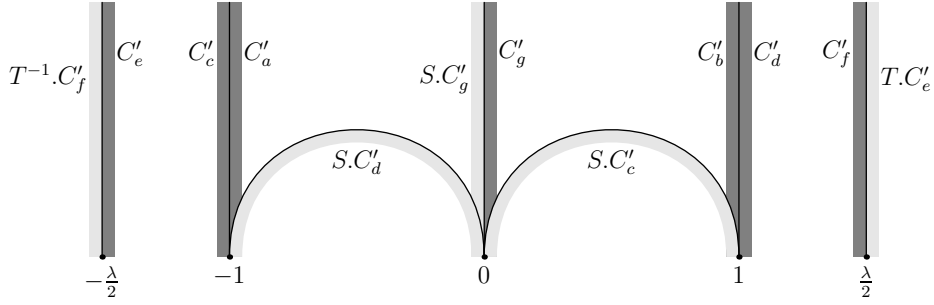}
\end{center}
\caption{Set of representatives $C'_P$ and future intersections.}
\label{setrep_original}
\end{figure}

The associated discrete dynamical system $(D_P,F_P)$ is defined on 
\begin{align*}
D_P &\sceq  \big((-1,\infty)_\st\times\{a\}\big)\ \cup\ \big((-\infty,1)_\st\times\{b\}\big)\ \cup\ \big((-\infty,-1)_\st\times\{c\}\big)
\\
& \quad \ \cup\ \big((1,\infty)_\st\times\{d\}\big)\ \cup\ \big((-\lambda/2,\infty)_\st\times\{e\}\big) 
\\
& \quad \ \cup\ \big((-\infty,\lambda/2)_\st\times\{f\}\big)\ \cup\ \big((0,\infty)_\st\times\{g\}\big)
\end{align*}
and given by the submaps
\begin{align*}
(-1,0)_\st\times\{a\} &\to (1,\infty)_\st\times\{d\}, & (x,a) &\mapsto (S.x,d)
\\
(0,\infty)_\st\times\{a\} &\to (0,\infty)_\st\times\{g\}, & (x,a) &\mapsto (x,g)
\\
(-\infty, 0)_\st\times\{b\} &\to (0,\infty)_\st\times\{g\}, & (x,b) &\mapsto (S.x,b)
\\
(0,1)_\st\times\{b\} &\to (-\infty,-1)_\st\times\{c\}, & (x,b) &\mapsto (S.x,c)
\\
\left(-\infty, -\tfrac{\lambda}{2}\right)_\st\times\{c\} &\to \left(-\infty, \tfrac{\lambda}{2}\right)_\st\times\{f\}, & (x,c) &\mapsto (T.x,f)
\\
\left(\tfrac{\lambda}{2},\infty\right)_\st\times\{d\} &\to \left(-\tfrac{\lambda}2,\infty\right)_\st\times\{e\}, & (x,d) &\mapsto (T^{-1}.x, e)
\\
(-1,\infty)_\st\times\{e\} &\to (-1,\infty)_\st\times\{a\}, & (x,e) & \mapsto (x,a)
\\
(-\infty, 1)_\st\times\{f\} &\to (-\infty, 1)_\st \times\{b\}, & (x,f) &\mapsto (x,b)
\\
(0,1)_\st\times\{g\} &\to (-\infty, -1)_\st\times\{c\}, & (x,g) &\mapsto (S.x, c)
\\
(1, \infty)_\st\times\{g\} &\to (1,\infty)_\st\times\{d\}, & (x,g) &\mapsto (x,d).
\end{align*}
Note that $(-\lambda/2,-1)_\st = \emptyset = (1,\lambda/2)_\st$, and thus
\[
 \left(-\infty,-\tfrac{\lambda}2\right)_\st = (-\infty, -1)_\st, \quad \left(\tfrac{\lambda}2,\infty\right)_\st = (1,\lambda)_\st
\]
as well as 
\[
 (-1,\infty)_\st = \left(-\tfrac{\lambda}2,\infty\right)_\st, \quad (-\infty, 1)_\st = \left(-\infty, \tfrac{\lambda}2\right)_\st.
\]
Hence $F_P$ is indeed defined on all of $D_P$.

The following proposition is essentially induced by the fact that the boundary points of the analytic hulls of the domains of the submaps of $(D_P,F_P)$ are not fixed by hyperbolic elements in $\Gamma$.

\begin{prop}[\cite{Pohl_Symdyn2d}]\label{bij_first}
The equivalence classes of the $F_P$-periodic orbits are in bijection with the conjugacy classes of the hyperbolic elements in $\Gamma$.
\end{prop}

\subsection{Reduction of the symbolic dynamics for the geodesic flow and associated transfer operator family}

Figure~\ref{setrep_original} and the discrete dynamical system $(D_P, F_P)$ show immediately that each periodic geodesic on $X$ which intersects  e.g.\@ $\pi(C'_e)$ also intersects $\pi(C'_a)$. Therefore, already a subset of $\wh C_P$ serves as a cross section. In the following lemma we provide a maximally reduced sub-cross section of $\wh C_P$.

\begin{lemma}
The subset 
\[
 \wh C_R \sceq \pi(C'_a \cup C'_b)
\]
of $\wh C_P$ is a cross section for the geodesic flow on $X$ with $C'_R \sceq C'_a\cup C'_b$ as set of representatives.
\end{lemma}

\begin{proof}
By inspecting the discrete dynamical system $(D_P,F_P)$ or by considering the iterated future intersections as indicated in Figure~\ref{iterated1} and \ref{iterated2}, one sees immediately that each periodic geodesic on $X$ lifts to a geodesic on $\h$ which intersects $C'_a \cup C'_b$.
\begin{figure}[h]
\begin{center}
\includegraphics{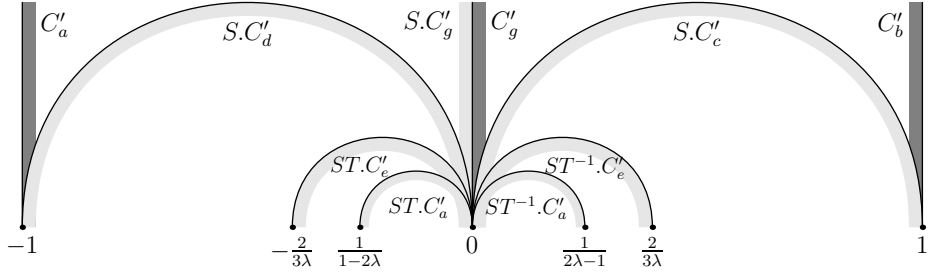}
\end{center}
\caption{Iterated future intersections: the part about $0$.}
\label{iterated1}
\end{figure}
\begin{figure}[h]
\begin{center}
\includegraphics{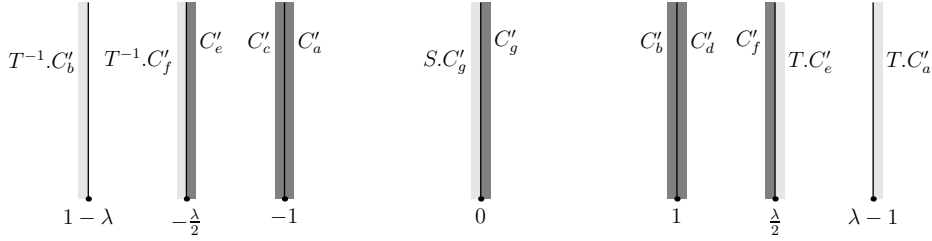}
\end{center}
\caption{Iterated future intersections: the part towards $\infty$.}
\label{iterated2}
\end{figure}

\end{proof}

The discrete dynamical system $(D_R,F_R)$ induced by $\wh C_R$ and  its set of representatives $C'_R$ is defined on 
\[
 D_R \sceq \big( (-1,\infty)_\st\times\{a\}\big) \ \cup\ \big( (-\infty,1)_\st\times\{b\}\big)
\]
and decomposes into the submaps
\begin{align*}
(-1,0)_\st\times\{a\} &\to (-1,\infty)_\st\times\{a\}, & (x,a) &\mapsto (T^{-1}S.x, a),
\\
(0,1)_\st\times\{a\} & \to (-\infty,1)_\st\times\{b\}, & (x,a)& \mapsto (TS.x, b),
\\
(-1+\lambda,\infty)_\st\times\{a\} & \to (-1,\infty)_\st\times\{a\}, & (x,a) &\mapsto (T^{-1}.x, a),
\\
(-\infty,1-\lambda)_\st\times\{b\} & \to (-\infty,1)_\st\times\{b\}, & (x,b) &\mapsto (T.x,b),
\\
(-1,0)_\st\times\{b\} & \to (-1,\infty)_\st\times\{a\}, & (x,b) &\mapsto (T^{-1}S.x,a),
\\
(0,1)_\st\times\{b\} & \to (-\infty, 1)_\st\times\{b\}, &(x,b) &\mapsto (TS.x,b).
\end{align*}
We remark that
\[
 (0,-1+\lambda)_\st = (0,1)_\st \quad\text{and}\quad (1-\lambda,0)_\st = (-1,0)_\st
\]
so that $F_R$ is defined on all of $D_R$. To simplify notation we set
\[
 g_1 \sceq T,\quad g_2\sceq T^{-1}S = \bmat{\lambda}{1}{-1}{0} \quad\text{and}\quad g_3\sceq TS = \bmat{\lambda}{-1}{1}{0}.
\]
The associated family of transfer operators 
\[
 \mc L_{R,s}\colon \Fct(D_R;\C) \to \Fct(D_R;\C)
\]
is then given by
\begin{align*}
\mc L_{R,s}f &= 1_{(-1,\infty)_\st\times\{a\}} \cdot \Big[ \tau_s(g_2) \big(f\cdot 1_{(-1,0)_\st\times\{a\}}\big) + \tau_s(g_1^{-1}) \big( f\cdot 1_{(-1+\lambda,\infty)_\st\times\{a\}}\big)
\\
& \qquad\qquad\qquad\qquad  + \tau_s(g_2) \big( f\cdot 1_{(0,1)_\st\times\{b\}}\big)\Big]
\\
& + 1_{(-\infty,1)_\st\times\{b\}} \cdot \Big[ \tau_s(g_3)\big(f\cdot 1_{(0,1)_\st\times\{a\}}\big) + \tau_s(g_1) \big(f\cdot 1_{(-\infty, 1-\lambda)_\st\times\{b\}}\big) 
\\
& \qquad\qquad\qquad\qquad  + \tau_s(g_3) \big(f\cdot 1_{(0,1)_\st\times\{b\}}\big)\Big].
\end{align*}
For $f\in\Fct(D_R;\C)$ we let
\[
 f_1 \sceq f \cdot 1_{ (-1,\infty)_\st\times\{a\}} \quad\text{and}\quad f_2 \sceq f \cdot 1_{ (-\infty,1)_\st\times\{b\}}.
\]
Thus $f=f_1 + f_2$. We may identify $f$ with the vector
\[
\begin{pmatrix}
f_1
\\
f_2
\end{pmatrix},
\]
and then $(-1,\infty)_\st\times\{a\}$ with $(-1,\infty)_\st$ and $(-\infty,1)_\st\times\{b\}$ with $(-\infty,1)_\st$. Then the transfer operator $\mc L_{R,s}$ has the matrix representation
\[
\mc L_{R,s} = 
\begin{pmatrix}
\tau_s(g_2) + \tau_s(g_1^{-1}) & \tau_s(g_2) 
\\
\tau_s(g_3) &  \tau_s(g_1) + \tau_s(g_3)
\end{pmatrix}.
\]

\subsection{Induction of the symbolic dynamics for the geodesic flow and associated transfer operator family}

Since the elements 
\[
 g_1 = \bmat{1}{\lambda}{0}{1} \quad\text{and}\quad g_1^{-1} = \bmat{1}{-\lambda}{0}{1}
\]
are parabolic and act on the diagonal of $\mc L_{R,s}$, this operator is not nuclear (on any nonzero domain of definition) and hence does not have a Fredholm determinant. 

To construct a closely related family of nuclear transfer operators, we ``accelerate'' the cross section $\wh C_R$ and the discrete dynamical system $(D_R,F_R)$ on these parabolic elements. Note that the elements $g_2$ and $g_3$ are hyperbolic. 

Let 
\begin{align*}
\mc N_R &\sceq \{ v\in C'_a \mid \gamma_v(-\infty) < -1-\lambda,\ \gamma_v(\infty)> -1+\lambda\}
\\
& \quad\quad\cup \{v\in C'_b \mid \gamma_v(\infty)<1-\lambda,\ \gamma_v(-\infty)>1+\lambda\}
\end{align*}
be the set of unit tangent vectors in $C'_R$ which cause consecutive applications of $T$ or $T^{-1}$ on the level of $(D_R, F_R)$. We set 
\[
\wh {\mc N}_R \sceq \pi(\mc N_R)
\]
as well as
\[
 \wh C_I \sceq \wh C_R\setminus \wh{\mc N}_R \quad\text{and}\quad C'_I \sceq C'_R\setminus \mc N_R.
\]

\begin{lemma}
The set $\wh C_I$ is a cross section for the geodesic flow on $X$ with $C'_I$ as set of representatives. The set $C'_I$ decomposes into the disjoint subsets
\[
 C'_{I,a} \sceq C'_a\setminus \mc N_R\quad\text{and}\quad  C'_{I,b}\sceq C'_b\setminus \mc N_R.
\]
\end{lemma}

\begin{proof}
One immediately checks that each periodic geodesic on $X$ intersects the set $\wh C_I$.
\end{proof}

The discrete dynamical system $(D_I, F_I)$ associated to the cross section $\wh C_I$ and its set of representatives $C'_I$ can now be deduced from $(D_R, F_R)$ by iterating $F_R$ on the submaps 
\begin{align*}
(-1+\lambda,\infty)_\st\times\{a\} & \to (-1,\infty)_\st\times\{a\}, & (x,a) &\mapsto (g_1^{-1}.x, a)
\intertext{and}
(-\infty,1-\lambda)_\st\times\{b\} & \to (-\infty,1)_\st\times\{b\}, & (x,b) &\mapsto (g_1.x,b)
\end{align*}
until these do not map into $(-1+\lambda,\infty)_\st\times\{a\}$ respectively $(-\infty, 1-\lambda)_\st\times\{b\}$.

The arising discrete dynamical system $(D_I, F_I)$ is defined on 
\[
 D_I\sceq \big((-1,\infty)_\st\times \{a\}\big) \cup \big( (-\infty, 1)_\st\times\{b\}\big)
\]
(which coincides with $D_R$) and given by the submaps
\begin{align*}
(-1,0)_\st\times\{a\} &\to (-1,\infty)_\st\times\{a\}, & (x,a) &\mapsto (g_2.x, a),
\\
(0,1)_\st\times\{a\} & \to (-\infty,1)_\st\times\{b\}, &(x,a)& \mapsto (g_3.x, b),
\\
(-1,0)_\st\times\{b\} & \to (-1,\infty)_\st\times\{a\}, & (x,b) &\mapsto (g_2.x,a),
\\
(0,1)_\st\times\{b\} & \to (-\infty, 1)_\st\times\{b\}, & (x,b) &\mapsto (g_3.x,b),
\end{align*}
and, for $n\in\N$,
\begin{align*}
(-1+n\lambda,-1+(n+1)\lambda)_\st\times\{a\} & \to (-1,-1+\lambda)_\st\times\{a\}, &(x,a)&\mapsto (g_1^{-n}.x, a),
\\
(1-(n+1)\lambda,1-n\lambda)_\st\times\{b\} & \to (1-\lambda,1)_\st\times\{b\}, &(x,b)&\mapsto (g_1^n.x,b).
\end{align*}

\begin{prop}\label{bij_final}
The equivalence classes of the $F_I$-periodic orbits are in bijection with the conjugacy classes of the hyperbolic elements in $\Gamma$.
\end{prop}

\begin{proof}
The relation between $F_P$, $F_R$ and $F_I$ implies that the equivalence classes of their respective periodic orbits are in bijection. Then the claim follows from Proposition~\ref{bij_first}.
\end{proof}

Formally, the transfer operator with parameter $s\in\C$ associated to $(D_I, F_I)$ is given by
\begin{align*}
\mc L_{I,s}f & = 1_{(-1,\infty)_\st\times\{a\}} \cdot \Big[ \tau_s(g_2)\big(f\cdot 1_{(-1,0)_\st\times\{a\}}\big) + \tau_s(g_2)\big(f\cdot 1_{(-1,0)_\st\times\{b\}}\big)\Big]
\\
& + 1_{(-\infty,1)_\st\times\{b\}} \cdot \Big[ \tau_s(g_3)\big(f\cdot 1_{(0,1)_\st\times\{a\}}\big) + \tau_s(g_3)\big(f\cdot 1_{(0,1)_\st\times\{b\}}\big)\Big]
\\
& + 1_{(-1,-1+\lambda)_\st\times\{a\}} \cdot \sum_{n\in\N} \tau_s(g_1^{-n})\big(f\cdot 1_{(-1+\lambda,\infty)_\st\times\{a\}}\big)
\\
& + 1_{(1-\lambda,1)_\st\times\{b\}} \cdot \sum_{n\in\N}\tau_s(g_1^n)\big(f\cdot 1_{(-\infty, 1-\lambda)_\st\times\{b\}}\big). 
\end{align*}
We will provide a convenient domain of definition in Section~\ref{sec:domain} below. As preparation, we define
\begin{align*}
 D_{1,\st} \sceq (-1,1)_\st, \quad D_{2,\st} \sceq (-1+\lambda,\infty)_\st, \quad D_{3,\st} \sceq (-\infty, 1-\lambda)_\st.
\end{align*}
Analogous to above, for any function $f\colon D_I \to \C$ we set
\begin{align*}
 f_1 &\sceq f \cdot 1_{D_{1,\st}\times\{a\}}, & f_2 &\sceq f\cdot 1_{D_{2,\st}\times\{a\}},
\\
f_3 & \sceq f\cdot 1_{D_{3,\st}\times\{b\}}, & f_4 & \sceq f\cdot 1_{D_{1,\st}\times\{b\}},
\end{align*}
identify $f$ with the vector
\[
\begin{pmatrix}
f_1 
\\
f_2
\\
f_3
\\
f_4
\end{pmatrix}
\]
and may omit the components $\{a\}$ respectively $\{b\}$ from the domains. Then the (formal) transfer operator $\mc L_{I,s}$ has the matrix representation
\[
\mc L_{I,s} = 
\begin{pmatrix}
\tau_s(g_2) & \sum_{n\in\N}\tau_s(g_1^{-n}) & 0 & \tau_s(g_2) 
\\
\tau_s(g_2) & 0 & 0 & \tau_s(g_2) 
\\
\tau_s(g_3) & 0 & 0 & \tau_s(g_3)
\\
\tau_s(g_3) & 0 & \sum_{n\in\N}\tau_s(g_1^n) & \tau_s(g_3)
\end{pmatrix}.
\]

\section{Domain of definition, nuclearity, and meromorphic continuation}\label{sec:domain}

We now construct a Banach space $B$ of holomorphic functions on some set $\mc D\subseteq \C$ such that, for $\Rea s > 1/2$, the operator $\mc L_{I,s}$ is nuclear of order $0$ on $B$, and the map $s\mapsto \mc L_{I,s}$ extends (in a weak sense) meromorphically to all of $\C$.  By the latter we mean here that there exist a discrete set $P\subseteq \C$ and for each $s\in\C\setminus P$ a nuclear operator $\wt{\mc L}_{I,s}\colon B\to B$ of order $0$ which equals $\mc L_{I,s}$ for $\Rea s > 1/2$. Furthermore, for each $f\in B$ and each $z\in \mc D$, the map $s\mapsto \wt{\mc L}_{I,s}f(z)$ is meromorphic with poles in $P$, and the map $(s,z) \mapsto \wt{\mc L}_{I,s}f(z)$ is continuous on $(\C\setminus P) \times \mc D$. 

During this construction we have to choose a neighborhood of $\infty$ in $P^1(\C)$. For simplicity, we conjugate our setup with \[
 \mc C \sceq \bmat{0}{1}{-1}{\frac{\lambda}{2}} \quad \in \PSL_2(\R).
\]
We call $\Gamma_C \sceq \mc C \Gamma \mc C^{-1}$ the conjugate discrete subgroup, and $(D_C, F_C)$ the conjugation of the discrete dynamical system $(D_I, F_I)$. For any subset $E\subseteq \R\cup\{\infty\}$ let 
\[
 E_\st'\sceq E\setminus \bd_C,
\]
where $\bd_C \sceq \mc C(\bd)$. For $j\in\{1,2,3\}$ let 
\[
 h_j \sceq \mc C g_j \mc C^{-1} \quad \text{and} \quad  E_{j,\st'} \sceq \mc C(D_{j,\st}).
\]
Thus,
\begin{align*}
h_1 = \bmat{1}{0}{-\lambda}{1}, \quad h_2 = \bmat{-\frac\lambda2}{1}{-\frac34\lambda^2-1}{\frac32\lambda}, \quad h_3  = \bmat{-\frac\lambda2}{1}{\frac{\lambda^2}{4}-1}{-\frac\lambda2},
\end{align*}
and
\begin{align*}
E_{1,\st'} = \left(\frac{2}{\lambda+2}, \frac{2}{\lambda-2}\right)_{\st'}, 
\quad
E_{2,\st'}  = \left(\frac{2}{2-\lambda},0\right)_{\st'}, 
\quad
E_{3,\st'}  = \left( 0, \frac{2}{3\lambda-2}\right)_{\st'}.
\end{align*}
Then 
\[
 D_C = \big(E_{1,\st'} \times \{a\}\big) \cup \big( E_{2,\st'} \times \{a\}\big) \cup \big( E_{3,\st'} \times \{b\}\big) \cup \big( E_{1,\st'} \times \{b\}\big),
\]
and the submaps of $(D_C, F_C)$ can easily be read off from $(D_I, F_I)$. Analogous to before, we identify any function $f\colon D_C \to \C$ with the vector $(f_1,\ldots, f_4)^\top$ where $f_1,f_4$ are defined on $E_{1,\st'}$, $f_2$ on $E_{2,\st'}$, and $f_3$ on $E_{3,\st'}$. Then the (formal) transfer operator $\mc L_{C,s}$ associated to $(D_C, F_C)$ has the matrix representation
\[
\mc L_{C,s} =
\begin{pmatrix}
\tau_s(h_2) &  \sum_{n\in\N}\tau_s(h_1^{-n}) & 0 & \tau_s(h_2) 
\\
\tau_s(h_2) & 0 & 0 & \tau_s(h_2) 
\\
\tau_s(h_3) & 0 & 0 & \tau_s(h_3)
\\
\tau_s(h_3) & 0 & \sum_{n\in\N}\tau_s(h_1^n) &  \tau_s(h_3)
\end{pmatrix}.
\]
Let 
\begin{align*}
E_1 &\sceq \left(\frac{2(\lambda-1)}{2+(\lambda-1)\lambda},\frac{2(1-\lambda)}{2+\lambda(1-\lambda)}\right),
\\
E_2 &\sceq \left(\frac{2}{2-\lambda}, 0\right),
\\
\text{and}\quad
E_3 &\sceq \left(0,\frac{2}{3\lambda-2}\right).
\end{align*}
The set $E_j$ is the analytic hull of $E_{j,\st'}$, and, up to problems of convergence, $\mc L_{C,s}$ acts on the function vectors
\[
 f = 
\begin{pmatrix}
f_1\colon E_1 \to\C
\\
f_2\colon E_2 \to \C
\\
f_3\colon E_3\to \C
\\
f_4\colon E_1\to \C
\end{pmatrix}.
\]

One might find both statements more obvious when considering the original system $(D_I, F_I)$. Here, the set $E_j$ corresponds to $D_j$, where 
\[
 D_1 \sceq \left(\frac{1}{1-\lambda},\frac{1}{-1+\lambda}\right), \quad D_2\sceq (-1+\lambda,\infty), \quad D_3\sceq (-\infty, 1-\lambda).
\]

Let 
\[
 \mc T\sceq \mc C J \mc C^{-1} = \bmat{1}{0}{\lambda}{-1}.
\]
We set 
\begin{align*}
 \mc E_1 &\sceq \left\{ z\in\C\left\vert\ \left| z-\frac{2\lambda}{(\lambda+2)(\lambda-2)}\right| < \frac{4}{(\lambda-2)(\lambda+2)} \right.\right\},
\\
\mc E_2 & \sceq \left\{ z\in\C \left\vert\ \left| z-\frac{\lambda-1}{\lambda(2-\lambda)}\right| < \frac{3-2\lambda}{\lambda(2-\lambda)}\right.\right\}
\\
\mc E_3 & \sceq \mc T.\mc E_2.
\end{align*}
This means that $\mc E_1$ and $\mc E_2$ are open discs in $\C$ with centers on $\R$ such that the boundary of $\mc E_1$ contains $\frac{2}{\lambda+2}$ and $\frac{2}{\lambda-2}$, and the boundary of $\mc E_2$ contains $\frac{3}{2-\lambda}$ and $\frac{1}{\lambda}$. Under application of $\mc C^{-1}$, $\mc E_1$ corresponds to 
\[
 \mc D_1 \sceq \{ |z|< 1\}
\]
and $\mc E_2$ corresponds to the open disc $\mc D_2$ in $\C\cup\{\infty\}$ with center on $\R$ which contains $\infty$ and whose boundary passes through $\frac{5\lambda-4}{6}$ and $-\frac{\lambda}{2}$.

For $j=1,2,3$ we let
\[
 B(\mc E_j) \sceq \{ \text{$f\colon \mc E_j \to \C$ holomorphic}\ \mid\ \text{$f$ extends continuously to $\overline{\mc E_j}$}\}.
\]
Endowed with the supremum norm, the space $B(\mc E_j)$ is a Banach space. Further we define
\[
 B(\mc E) \sceq B(\mc E_1) \times B(\mc E_2) \times B(\mc E_3) \times B(\mc E_1)
\]
to be the direct product of the previous Banach spaces.

\begin{thm}\label{nuclear}
\begin{enumerate}[{\rm (i)}]
\item\label{nucleari} For $\Rea s > \frac12$, the transfer operator $\mc L_{C,s}$ is a self-map of $B(\mc E)$ and nuclear of order $0$.
\item The map $s\mapsto \mc L_{C,s}$ extends meromorphically to all of $\C$. The possible poles are located at $s=(1-k)/2$, $k\in\N_0$, and are of order $\leq 1$. For each pole $s_0$, there is a neighborhood $U$ of $s_0$ such that the meromorphic extension $\wt{\mc L}_{C,s}$ is of the form 
\[
 \wt{\mc L}_{C,s} = \frac1{s-s_0} \mc A_s + \mc B_s
\]
where the operators $\mc A_s$ and $\mc B_s$ are holomorphic on $U$, and $\mc A_s$ is of rank at most $2$.
\item The Fredholm determinant $s\mapsto \det(1-\wt{\mc L}_{C,s})$ is a meromorphic function on $\C$ with possible poles located at $s=(1-k)/2$, $k\in\N_0$. The order of a pole is at most $2$.
\end{enumerate}
\end{thm}

\begin{proof}
We observe that 
\begin{itemize}
\item $\overline E_1 \subseteq\mc E_1$, $\overline E_2\subseteq \mc E_2$, $\overline E_3 \subseteq \mc E_3$,
\item $\mc T.\mc E_1 = \mc E_1$, $\mc T.\mc E_2 = \mc E_3$,
\item $h_2^{-1}.\overline{\mc E_1} \subseteq \mc E_1$, $h_2^{-1}.\overline{\mc E_2}\subseteq \mc E_1$, and $h_1^n.\overline{\mc E_1} \subseteq \mc E_2$ for all $n\in\N$.
\end{itemize}
Then the proof of this theorem is a straighforward adaption of \cite{Ruelle_zeta, Mayer_thermo, Moeller_Pohl}.
\end{proof}

\section{The Selberg zeta function as the Fredholm determinant of $\wt{\mc L}_{I,s}$}

We return to the system $(D_I, F_I)$ and consider, for $\Rea s > 1/2$, the formal transfer operator $\mc L_{I,s}$ as an actual operator on $B(\mc D_1)\times B(\mc D_2) \times B(\mc D_3) \times B(\mc D_1)$ with $\mc D_3 \sceq J.\mc D_2$. We denote by $s\mapsto \wt{\mc L}_{I,s}$ the meromorphic extension of $s\mapsto \mc L_{I,s}$ to all of $\C$. 

\begin{thm}\label{fredholm}
We have $Z(s) = \det(1-\wt{\mc L}_{I,s})$. More precisely, for $\Rea s >\delta$, we have $Z(s) = \det(1-\mc L_{I,s})$, and the right hand side (and thus also the left hand side) extends meromorphically to all of $\C$ with possible poles of order at most $2$ at $s=(1-k)/2$, $k\in\N_0$.
\end{thm}

\begin{proof}
By Theorem~\ref{nuclear}\eqref{nucleari}, for $\Rea s > 1/2$, the operators $\mc L_{I,s}$ have a Fredholm determinant. Since 
the equivalence classes of $F_I$-periodic orbits are in bijection with the conjugacy classes of the hyperbolic elements in $\Gamma$ (Proposition~\ref{bij_final}), the thermodynamic formalism (cf.\@ Section~\ref{sec:thermo}) now implies $Z(s) = \det(1-\mc L_{I,s})$ for $\Rea s > \delta$. Theorem~\ref{nuclear} completes the proof. 
\end{proof}

We remark that the proof of Theorem~\ref{fredholm} does not rely on the already known existence of a meromorphic continuation of the Selberg zeta function \cite{Guillope}. It rather constitutes an alternative proof of this fact.

\section{Billiard flow}\label{sec:billiard}

The element $J = \textbmat{-1}{0}{0}{1}$ commutes with the Hecke triangle group $\Gamma$. The extended  discrete group
\[
 \wt\Gamma \sceq \langle \Gamma, J\rangle = \Gamma \cup J\Gamma \quad \leq \PGL_2(\R)
\]
is the triangle group underlying $\Gamma$. A fundamental domain for its action on $\h$ is indicated in Figure~\ref{funddomtriangle}.

\begin{figure}[h]
\begin{center}
\includegraphics{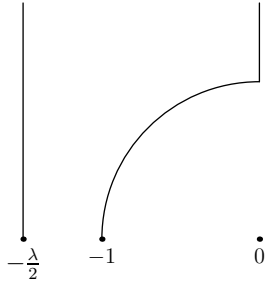}
\end{center}
\caption{Fundamental domain for $\wt\Gamma$.}
\label{funddomtriangle}
\end{figure}

In this section we consider the billiard flow on $\wt\Gamma\backslash\h$ and use the discretizations from Section~\ref{sec:systems} for the geodesic flow on $\Gamma\backslash\h$ to provide weighted discretizations for the billiard flow. The Fredholm determinants of the two associated families of weighted transfer operators define dynamical zeta functions which give a factorization of the Selberg zeta function for the geodesic flow. We suggest to think about the weighting and the factorization as encoding the Dirichlet resp.\@ Neumann boundary conditions for the billiard flow and comment in Section~\ref{sec:outlook} below on the conjectural relation.

Recall the map $\pi\colon S\h \to \Gamma\backslash S\h$ and define $\pi_J\colon S\h \to \wt\Gamma\backslash S\h$ to be the canonical quotient map. 
The cross sections $\wh C_R$ and $\wh C_I$ from Section~\ref{sec:systems} for the geodesic flow are obviously invariant under the action of $J$, and we have $J.C'_a = C'_b$. The following lemma is then immediate.

\begin{lemma}
The sets 
\[
 \wh C_{R^J} \sceq \pi_J\big( \pi^{-1}(\wh C_R) \big) \quad\text{and}\quad \wh C_{I^J} \sceq \pi_J \big( \pi^{-1}(\wh C_I) \big)
\]
are cross sections for the billiard flow on $\wt\Gamma\backslash\h$ with sets of representatives $C'_{R^J} \sceq C'_a$ resp.\@  $C'_{I^J} \sceq C'_{I,a}$. 
\end{lemma}

We endow the arising discrete dynamical systems $(D_{R^J}, F_{R^J})$ (for $(\wh C_{R^J}, C'_{R^J})$) and $(D_{I^J}, F_{I^J})$ (for $(\wh C_{I^J}, C'_{I^J})$) with two different weight functions, which essentially only effect the submaps with an acting element from $\wt\Gamma\setminus \Gamma$. The weighted systems $(D_{R^J}, F_{R^J}, \pm)$ are given by the submaps
\begin{align*}
(0,1)_\st & \to (-1,\infty)_\st, & x & \mapsto g_2.x, && \text{weight: $1$}
\\
(0,1)_\st & \to (-1,\infty)_\st, & x & \mapsto g_2J.x, && \text{weight: $\pm 1$}
\\
(\lambda-1,\infty)_\st & \to (-1,\infty)_\st, & x & \mapsto g_1^{-1}.x, && \text{weight: $1$.}
\end{align*}
We refer to Figure~\ref{nexttriangle} for the location of the future intersections and the relation between $C'_{R^J}$ and $C'_R$, which allows to read off these submaps and also shows how to deduce them directly from the submaps of $(D_R, F_R)$.

\begin{figure}[h]
\begin{center}
\includegraphics{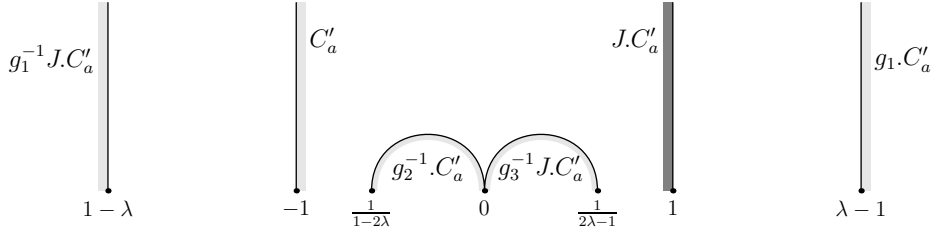}
\end{center}
\caption{Future intersections for $C'_{R^J}$ and relation to $C'_R$.}
\label{nexttriangle}
\end{figure}

The weighted systems $(D_{I^J}, F_{I^J}, \pm)$ are given by the submaps ($n\in\N$)
\begin{align*}
(-1,0)_\st & \to (-1,\infty)_\st, &x &\mapsto g_2.x, && \text{weight: $1$}
\\
(0,1)_\st & \to (-1,\infty)_\st, & x & \mapsto g_2J.x, && \text{weight: $\pm 1$}
\\
(-1+n\lambda, -1 + (n+1)\lambda)_\st & \to (-1,1)_\st, & x & \mapsto g_1^{-n}.x, && \text{weight: $1$.}
\end{align*}

The associated transfer operators include the weights as factors for the respective submaps. They are here
\[
 \mc L_{R^J,s}^\pm = \tau_s(g_2) + \tau_s(g_1^{-1}) \pm \tau_s(g_2J),
\]
acting on functions $D_{R^J} \to \C$, and 
\[
 \mc L_{I^J,s}^\pm =
\begin{pmatrix}
\tau_s(g_2) \pm \tau_s(g_2J) & \sum_{n\in\N}\tau_s(g_1^{-n})
\\
\tau_s(g_2) \pm \tau_s(g_2J) & 0 
\end{pmatrix},
\]
formally acting on functions 
\[
 f = 
\begin{pmatrix}
f_1\colon (-1,1)_\st \to\C
\\
f_2\colon (\lambda-1,\infty)_\st\to\C
\end{pmatrix}.
\]

\begin{thm}\label{fredtriangle}
For $\Rea s > 1/2$, the operators $\mc L_{I^J,s}^\pm$ are nuclear of order $0$ on $B(\mc D_1)\times B(\mc D_2)$. The maps $s\mapsto \mc L_{I^J,s}^\pm$ extend meromorphically to all of $\C$ with possible poles at $s=(1-k)/2$, $k\in\N_0$, of order at most $1$. Let $s\mapsto \wt{\mc L}_{I^J,s}^\pm$ denote the extended maps. Their Fredholm determinants 
\[
 Z_\pm(s) \sceq \det\left( 1 - \wt{\mc L}_{I^J,s}^\pm \right)
\]
define meromorphic functions on $\C$, and 
\[
 Z_+(s) Z_-(s) = \det\left( 1 - \wt{\mc L}_{I^J,s}^+ \right)\det\left( 1 - \wt{\mc L}_{I^J,s}^- \right) = \det\left( 1 - \wt{\mc L}_{I,s} \right) = Z(s).
\]
\end{thm}

\begin{proof}
Let 
\[
 \mc P_s \sceq \frac{1}{\sqrt{2}} 
\begin{pmatrix}
1 &  &  & \tau_s(J) 
\\
 & 1 & \tau_s(J) &
\\
 & \tau_s(J) & -1 &
\\
\tau_s(J) & & & -1
\end{pmatrix}
\]
and 
\[
 \mc R_s \sceq  \mat{}{\tau_s(J)}{\tau_s(J)}{}.
\]
Then $\mc P_s^2 = \id$, $\mc R_s^2 = \id$, and
\begin{equation}\label{conjugation}
 \mc P_s \mc L_{I,s} \mc P_s = \mat{ \mc L_{I^J,s}^+ }{}{}{ \mc R_s\mc L_{I^J,s}^-\mc R_s }.
\end{equation}
This, together with Theorem~\ref{nuclear} and \ref{fredholm}, proves all statements. 
\end{proof}

Since the Selberg zeta function $Z$ for $\Gamma$ does not vanish for $\Rea s > \delta\ (>1/2)$ (cf.\@ Section~\ref{sec:selberg}) and the transfer operators $\mc L_{I,s}$ are defined for $\Rea s > 1/2$, the thermodynamic formalism shows that
 \[
 -\sum_{\ell\in\N} \frac{1}{\ell} \Tr \mc L^\ell_{I,s} \quad \big( =\log \det(1-\mc L_{I,s}) = \log Z(s) \big)
\]
converges for $\Rea s > \delta$ (cf.\@ the proof of Theorem~\ref{fredholm}). By \eqref{conjugation} in the proof of Theorem~\ref{fredtriangle}, the analogous statements hold for the operators $\mc L_{I^J,s}^\pm$. This allows us to provide, in Theorem~\ref{explicit} below, explicit formulas for their Fredholm determinants in terms of the conjugacy classes of primitive hyperbolic elements in $\wt\Gamma$ and their norms, or equivalently, the length spectrum of the billiard flow. It shows that $Z_\pm$ are Selberg-like zeta functions.

For the statement of Theorem~\ref{explicit}, we recall that an element $h\in\wt\Gamma$ is called hyperbolic if and only if $h^2 \in \Gamma$ is hyperbolic. The norm of a hyperbolic element $h\in\wt\Gamma$ is defined as
\[
 N(h) \sceq N(h^2)^{1/2}.
\]
An element $h\in\wt\Gamma$ is called primitive if $h_0^n = h$ for $h_0\in\wt\Gamma$ and $n\in\Z$ implies $|n|=1$. We let 
\[
[h] = \{ ghg^{-1} \mid g\in\wt\Gamma\}
\]
denote the $\wt\Gamma$-conjugacy class of $h$. Further, we denote by $[\wt\Gamma]_p$ the set of $\wt\Gamma$-conjugacy classes of the primitive hyperbolic elements in $\wt\Gamma$, and by $[\wt\Gamma]_h$ the set of $\wt\Gamma$-conjugacy classes of the hyperbolic elements in $\wt\Gamma$.

\begin{thm}\label{explicit}
For $\Rea s > \delta$ we have
\begin{align*}
Z_+(s) & = \prod_{ [g] \in [\wt\Gamma]_p } \prod_{k=0}^\infty \left( 1 - \det g^k N(g)^{-(s+k)}\right)
\intertext{and}
Z_-(s) & = \prod_{ [g] \in [\wt\Gamma]_p } \prod_{k=0}^\infty \left( 1 - \det g^{k+1} N(g)^{-(s+k)}\right).
\end{align*}
\end{thm}

\begin{proof}
We start by investigating how $F_{I^J}$-periodic orbits are related to $\wt\Gamma$-conjugacy classes of the hyperbolic elements in $\wt\Gamma$. We recall from Proposition~\ref{bij_final} that the $\Gamma$-conjugacy classes of the hyperbolic elements in $\Gamma$ are in bijection with the equivalence classes of $F_I$-periodic orbits. This and \eqref{conjugation} (Theorem~\ref{fredtriangle}) implies that $[\wt\Gamma]_h$ is in bijection with the equivalence classes of $F_{I^J}$-periodic orbits. Instead of using \eqref{conjugation} one can deduce this fact also by analyzing the relation between hyperbolic elements in $\Gamma$ and in $\wt\Gamma$, and the relation between $F_I$-periodic orbits and $F_{I^J}$-periodic orbits. As for geodesic flows, we identify any $F_{I^J}$-periodic orbit $(x_k)_{k=1}^\ell$ with 
the sequence $(a_k)_{k=1}^\ell$ in $\{g_2, g_2J, g_1^{-n} \mid n\in\N \}$ of the acting elements in the submaps, that is  $a_k.x_k  = x_{k+1}$ for $k=1,\ldots, n-1$, and $a_n.x_n=x_1$. For $\ell\in\N$, we let $\Per_\ell$ denote the set of arising sequences $(a_k)_{k=1}^\ell$ of length $\ell$, and let $P_\ell$ be the set of elements $a\sceq a_1\cdots a_\ell$. The map
\[
 H\colon \bigcup_{\ell\in\N} \Per_\ell \to \bigcup_{\ell\in\N} P_\ell,\quad (a_k)_{k=1}^\ell \mapsto a_1\cdots a_\ell,
\]
is injective, its image consists of hyperbolic elements and contains at least one representative for each $\wt\Gamma$-conjugacy class. For a hyperbolic element $a\in\wt\Gamma$ let $n(a)$ denote the (unique) positive integer such that $a = a_0^{n(a)}$ for some (unique) primitive hyperbolic element $a_0\in\wt\Gamma$. Note that $n(a)$ is constant on the conjugacy class of $a$. For $[a]\in [\wt\Gamma]_h$ pick any representative $a\in\bigcup_{\ell\in\N} P_\ell$, and let $\ell(a)$ denote the length of its representing sequence $(a_k)_{k=1}^\ell$. Note that also $\ell(a)$ only depends on $[a]$. Then $m(a) \sceq \ell(a)/n(a)$ is integral and $(a_k)_{k=1}^\ell$ equals $n(a)$ consecutions of $(a_1,\ldots, a_{m(a)})$. Therefore, $[a]$ is represented by $m(a)$ elements in $P_{\ell(a)}$.

We set
\[
 b_s^+(a) \sceq \tau_s(a) \quad\text{and}\quad b_s^-(a) \sceq \det a \cdot \tau_s(a).
\]
When considered as acting on $B(\mc D_1)$ or $B(\mc D_2)$ (as appropriate), we have (see \cite{Pohl_spectral_hecke}) 
\[
 \Tr b_s^+(a) = \frac{N(a)^{-s}}{1- \det a\cdot N(a)^{-1}} \quad\text{and}\quad \Tr b_s^{-}(a) = \frac{\det a\cdot N(a)^{-s}}{1-\det a\cdot N(a)^{-1}}.
\]
Moreover, for $\Rea s >\delta$, we have
\[
 \Tr \big(\mc L_{I^J,s}^\pm\big)^\ell = \sum_{a\in P_\ell} \Tr b_s^\pm(a).
\]
Then
\begin{align*}
\log Z_-(s) & = - \sum_{[a]\in [\wt\Gamma]_h} \frac{1}{n(a)} \frac{\det a\cdot N(a)^{-s}}{1-\det a\cdot N(a)^{-1}}
\\
& = -\sum_{\ell\in\N} \frac{1}{\ell} \sum_{\stackrel{[a]\in [\wt\Gamma]_h}{\ell(a)=\ell}} \frac{\ell(a)}{n(a)} \frac{\det a\cdot N(a)^{-s}}{1-\det a\cdot N(a)^{-1}}
\\
& = -\sum_{\ell\in\N} \frac{1}{\ell} \sum_{a\in P_\ell} \frac{\det a\cdot N(a)^{-s}}{1-\det a\cdot N(a)^{-1}}
\\
& = - \sum_{\ell\in\N} \frac{1}{\ell} \sum_{a\in P_\ell} \Tr b_s^-(a)
\\
& = - \sum_{\ell\in\N} \frac{1}{\ell} \Tr \mc L^\ell_{I^J,s},
\end{align*}
which shows the claimed identity for $Z_-$. 
The proof for $Z_+$ is analogous.
\end{proof}

We remark that $Z_\pm$ are actually dynamical zeta functions. As for geodesic flows, the $\wt\Gamma$-conjugacy classes of hyperbolic element in $\wt\Gamma$ are (here) in bijection with the periodic billiards on $\wt\Gamma\backslash\h$. If $\wt\gamma$ is such a periodic billiard of length $\ell(\wt\gamma)$ and $h$ an associated hyperbolic element, then 
\[
 N(h) = \exp \ell(\wt\gamma).
\]
and $\det h$ is just the parity of the bounces of $\wt\gamma$ of the boundary of $\wt\Gamma\backslash\h$.

\section{A few conjectures}\label{sec:outlook}

In Section~\ref{sec:systems} we started with discretizations for the geodesic flow which give rise to finite-term transfer operators. For cofinite discrete Hecke triangle groups, the highly regular eigenfunctions with eigenvalue $1$ of the respective transfer operators are in bijection with the Maass cusp forms for these lattices \cite{Moeller_Pohl}. We expect that for non-cofinite Hecke triangle groups, a similar relation holds between the eigenfunctions of the transfer operators $\mc L_{R,s}$ (or $\mc L_{P,s}$) and the residues at the resonances $s$.

Furthermore, for the billiard flow we used the two weight functions $\pm 1$. For cofinite discrete Hecke triangle groups, these weights correspond to Neumann ($+$) resp.\@ Dirichlet ($-$) boundary conditions, and thus allow to separate the odd and the even spectrum. We expect that the same interpretation holds for non-cofinite Hecke triangle groups. A first hint towards such a result is contained in \cite{Guillope}, who, however, requires \textit{compact} totally geodesic boundaries. Moreover, we expect that the finite-term weighted transfer operators provide a kind of period functions for odd resp. even residues at resonances.


\providecommand{\bysame}{\leavevmode\hbox to3em{\hrulefill}\thinspace}
\providecommand{\MR}{\relax\ifhmode\unskip\space\fi MR }
\providecommand{\MRhref}[2]{%
  \href{http://www.ams.org/mathscinet-getitem?mr=#1}{#2}
}
\providecommand{\href}[2]{#2}

\end{document}